\documentclass{amsproc}
\usepackage{amsmath}
\usepackage{amsfonts}

\theoremstyle{plain}

\newtheorem{conjecture}{Conjecture}

\newtheorem{lemma}{Lemma}

\newtheorem{theorem}{Theorem}
\numberwithin{equation}{section}

\begin{document}
\title[The floating body and the hyperplane conjecture]{Comments on the floating body and the hyperplane conjecture}
\author{Daniel Fresen}
\address{Department of Mathematics, University of Missouri}
\email{djfb6b@mail.missouri.edu}
\dedicatory{To my dear sister Anna}
\thanks{Many thanks to Elizabeth Meckes, Mark Meckes, Mathieu Meyer and Elisabeth Werner for helpful comments, to John Fresen and Nigel Kalton for their advice and encouragement, and to my advisors Alexander Koldobsky and Mark Rudelson.}
\date{February 2011}

\begin{abstract}
We provide a reformulation of the hyperplane conjecture (the slicing problem) in terms of the floating body and give upper and lower bounds on the logarithmic Hausdorff distance between an arbitrary convex body $K\subset \mathbb{R}^{d}$\ and the convex
floating body $K_{\delta }$ inside $K$.
\end{abstract}

\maketitle

\section{Introduction}

Let $d\in \mathbb{N}$ and let $K\subset \mathbb{R}^{d}$ be a convex body (a
compact convex set with non-empty interior). \ We shall index half-spaces as 
$\mathfrak{H}_{\theta ,t}=\{x\in \mathbb{R}^{d}:\left\langle x,\theta
\right\rangle \geq t\}$ where $\theta \in S^{d-1}$ and $t\in \mathbb{R}$.
For any $\delta >0$, the \textit{convex floating body} inside $K$ is defined
as%
\[
K_{\delta }=\cap \{\mathfrak{H}_{\theta ,t}:\mathrm{vol}_{d}(\mathfrak{H}%
_{\theta ,t}\cap K)\geq (1-\delta )\mathrm{vol}_{d}(K)\}
\]%
\ We shall refer to $K_{\delta }$ simply as the floating body, although this
terminology is sometimes used for a non-convex variation of $K_{\delta }$. \
The convex floating body was introduced by Sch\"{u}tt and Werner \cite{SW}
and is a well studied object in convex geometry;\ it is related to
Gauss-Kronecker curvature, affine surface area and polyhedral approximation (%
\cite{SW}, \cite{Sc1}, \cite{Sc2}) and plays an important role in the study
of random polytopes (see e.g. \cite{BL} and \cite{Vu} p. 1290). The
definition of the convex floating body can be extended to an arbitrary
probability measure on $\mathbb{R}^{d}$ in the obvious way (see \cite{Fr})
and is a very natural multivariate version of a quantile. In this paper we
discuss a relationship between the floating body and the hyperplane
conjecture.

The hyperplane conjecture (also known as the slicing problem) speculates
that there exists a universal constant $c>0$ such that for any $d\in \mathbb{%
N}$ and any convex body $K\subset \mathbb{R}^{d}$ of unit volume, there
exists a hyperplane $\mathcal{H}\subset \mathbb{R}^{d}$ such that $\mathrm{%
vol}_{d-1}(\mathcal{H}\cap K)\geq c$. \ This conjecture goes back to the
1986 paper of Bourgain \cite{Bo0} (see the remark on p.1470) and is one of
the most fundamental unsolved problems in asymptotic convex geometry.\ \ It
is equivalent to several other open problems in the area (see \cite{MP}).
One such open problem is a variation of the Busemann-Petty problem \cite{GKS}%
:

\begin{conjecture}
There exists a universal constant $c>0$ with the following property: if $%
K,L\subset \mathbb{R}^{d}$ denote any centrally symmetric convex bodies in
any dimension $d$ such that for all central hyperplanes $H$ we have%
\[
\mathrm{vol}_{d-1}(K\cap H)\leq \mathrm{vol}_{d-1}(L\cap H)
\]%
then it follows that%
\[
\mathrm{vol}_{d}(K)\leq c\cdot \mathrm{vol}_{d}(L)
\]
\end{conjecture}

Let $X$ be a random vector uniformly distributed inside a convex body $K$. \
We say that $K$ is isotropic\footnote{%
Note that the word isotropic is used to mean several different things in the
literature. \ Some authors require $L_{K}=1$, others require $\mathrm{vol}%
_{d}(K)=1$. \ Our use of the word does not include either of these
conditions.} if its centroid lies at the origin and the covariance of $X$
obeys%
\[
\mathrm{cov}(X)=L_{K}^{2}I_{d}
\]%
where \thinspace $L_{K}>0$ is called the isotropic constant of $K$ and $%
I_{d} $ is the $d\times d$ identity matrix. \ Any convex body can be brought
to isotropic position via an affine map (see e.g. \cite{Ba0} or \cite{MP}).
It is well known that the hyperplane conjecture holds if and only if there
exists a universal constant $c>0$ such that $L_{K}<c$ for any isotropic
convex body of unit volume in any dimension. The best bound to date is $%
L_{K}<cd^{1/4}$ by Klartag \cite{Kl2} (see also \cite{Bo}). Let $%
B(0,r)=\{x\in \mathbb{R}^{d}:||x||_{2}\leq r\}$ and consider the following
two statements:

\begin{itemize}
\item $\Lambda _{1}$: there exists $\delta \in (0,e^{-1})$ and $r>0$ such
that for all $d\in \mathbb{N}$ and all isotropic convex bodies $K\subset 
\mathbb{R}^{d}$ of unit volume, $K_{\delta }\subset B(0,r)$.

\item $\Lambda _{2}$: for all $\delta \in (0,e^{-1})$ there exists $r>0$
such that for all $d\in \mathbb{N}$ and all isotropic convex bodies $%
K\subset \mathbb{R}^{d}$ of unit volume, $K_{\delta }\subset B(0,r)$.
\end{itemize}

\begin{theorem}
The hyperplane conjecture is equivalent to both $\Lambda _{1}$ and $\Lambda
_{2}$.
\end{theorem}

The proof of theorem 1 essentially comes down to an observation that due to
the rigidity of one dimensional log-concave probability distributions, the
variance of such a distribution is controlled by certain quantiles (and vice
versa). Thus the condition of bounded variance ($L_{K}<c$) equates to a
bound on the quantiles, and this easily transfers to the statements $\Lambda
_{1}$ and $\Lambda _{2}$ concerning the floating body.

We define the \textit{logarithmic Hausdorff distance} between two convex
bodies $K,L\subset \mathbb{R}^{d}$ as%
\[
d_{\mathfrak{L}}(K,L)=\inf \{\lambda \geq 1:\exists x\in int(K\cap L)\text{, 
}\lambda ^{-1}(K-x)+x\subset L\subset \lambda (K-x)+x\}
\]%
The logarithmic Hausdorff distance compares both the size and the shape of
the two bodies and is related to both the Hausdorff distance $d_{\mathcal{H}%
} $ and the Banach-Mazur distance $d_{BM}$. \ Unlike the Hausdorff distance,
it is invariant under affine transformations that act simultaneously on both
bodies, yet it is stronger than the Banach-Mazur distance which is blind to
affine transformations that act on one body but not the other. We end the
paper with the following two theorems. The bound (\ref{geometric}) is used
in \cite{Fr}.

\begin{theorem}
For all $d\in \mathbb{N}$, all convex bodies $K\subset \mathbb{R}^{d}$ and
all $\delta \leq 8^{-d}$, the body $K_{\delta }$ is non-empty and we have
the inequalities%
\begin{eqnarray}
d_{\mathfrak{L}}(K,K_{\delta }) &\leq &1+8\delta ^{1/d}  \label{geometric} \\
d_{BM}(K,K_{\delta }) &\leq &1+24\delta ^{1/d}  \label{Banach-Mazur}
\end{eqnarray}
\end{theorem}

By (\ref{simplex bound}), inequality (\ref{geometric}) is sharp (except for
the constant $8$) when $K$ is a simplex.

\begin{theorem}
There exists a universal constant $c>0$ such that for all $d\in \mathbb{N}$,
all convex bodies $K\subset \mathbb{R}^{d}$ and all $\delta \in (0,e^{-1})$,
the body $K_{\delta }$ is non-empty and 
\begin{equation}
d_{\mathfrak{L}}(K,K_{\delta })\geq \frac{cd^{\frac{1}{4}}}{\log (2\delta
^{-1})}  \label{lower bound 2}
\end{equation}
\end{theorem}

\bigskip

\section{Preliminaries}

For any convex body $K\subset \mathbb{R}^{d}$ containing the origin as an
interior point, we consider the dual Minkowski functional $||\cdot
||_{K^{\circ }}$ defined by $||y||_{K^{\circ }}=\sup \{\left\langle
y,x\right\rangle :x\in K\}$. \ Let $h_{K}:S^{d-1}\rightarrow \mathbb{R}$
denote the support function of $K$, which is the restriction of $||\cdot
||_{K^{\circ }}$ to $S^{d-1}$. For any $\theta \in S^{d-1}$ and $t\in 
\mathbb{R}$, let $\mathcal{H}_{\theta ,t}=\{y\in \mathbb{R}^{d}:\left\langle
\theta ,y\right\rangle =t\}$.\ Let $\psi _{K,\theta }$ denote the normalized
parallel section function of $K$ in the direction of $\theta $,%
\[
\psi _{K,\theta }(t)=\frac{\mathrm{vol}_{d-1}(K\cap \mathcal{H}_{\theta ,t})%
}{\mathrm{vol}_{d}(K)}
\]%
and define%
\[
A_{K,\theta }(t)=\int_{t}^{h_{K}(\theta )}\psi _{K,\theta }(s)ds
\]%
where in both formulae $t\in \lbrack -h_{K}(-\theta ),h_{K}(\theta )]$.\ If $%
X$ is a random vector uniformly distributed in $K$, then $\psi _{K,\theta }$
is the density function of $\left\langle X,\theta \right\rangle $. We denote
the median of this random variable by%
\[
m_{K,\theta }=A_{K,\theta }^{-1}(1/2)
\]%
The \textit{Banach-Mazur distance }between convex bodies $K$ and $L$ shall
be denoted by $d_{BM}(K,L)$ and is defined as,%
\[
d_{BM}(K,L)=\inf \{\lambda \geq 1:\exists x\in \mathbb{R}^{d},\exists
T,K\subset TL\subset \lambda (K-x)+x\}
\]%
where $T$ represents an affine transformation of $\mathbb{R}^{d}$. By
convexity we can express the Hausdorff distance between $K$ and $L$ as%
\[
d_{\mathcal{H}}(K,L)=\sup_{\theta \in S^{d-1}}\left\vert \sup_{x\in
K}\left\langle \theta ,x\right\rangle -\sup_{x\in L}\left\langle \theta
,x\right\rangle \right\vert
\]%
We define the \textit{logarithmic Hausdorff distance} between $K$ and $L$
about a point $x\in int(K\cap L)$ as%
\[
d_{\mathfrak{L}}(K,L,x)=\inf \{\lambda \geq 1:\lambda ^{-1}(K-x)+x\subset
L\subset \lambda (K-x)+x\}
\]%
and%
\[
d_{\mathfrak{L}}(K,L)=\inf \{d_{\mathfrak{L}}(K,L,x):x\in int(K\cap L)\}
\]%
Note that%
\[
\log d_{\mathfrak{L}}(K,L,0)=\sup_{\theta \in S^{d-1}}\left\vert \log
||\theta ||_{K}-\log ||\theta ||_{L}\right\vert
\]%
The following relations follow directly from the definitions above,%
\begin{eqnarray*}
d_{\mathfrak{L}}(K,L,0) &=&d_{\mathfrak{L}}(K^{\circ },L^{\circ },0) \\
d_{BM}(K,L) &\leq &d_{\mathfrak{L}}(K,L)^{2}
\end{eqnarray*}

\bigskip

\section{The hyperplane conjecture}

A function $f:\mathbb{R}^{d}\rightarrow \lbrack 0,\infty )$ is log-concave
(see \cite{KM} and \cite{LV}) if for any $x,y\in \mathbb{R}^{d}$ and any $%
\lambda \in (0,1)$ we have%
\[
f(\lambda x+(1-\lambda )y)\geq f(x)^{\lambda }f(y)^{1-\lambda }
\]%
The support of such a function will necessarily be convex, and $-\log f$ is
a convex function. A probability measure is log-concave if for any compact
sets $\Omega _{1},\Omega _{2}\subset \mathbb{R}^{d}$ and any $\lambda \in
(0,1)$ the following inequality holds,%
\[
\mu (\lambda \Omega _{1}+(1-\lambda )\Omega _{2})\geq \mu (\Omega
_{1})^{\lambda }\mu (\Omega _{2})^{1-\lambda }
\]%
where $\Omega _{1}+\Omega _{2}=\{x+y:x\in \Omega _{1},y\in \Omega _{2}\}$ is
the Minkowski sum of $\Omega _{1}$ and $\Omega _{2}$. The Brunn-Minkowski
inequality (in its multiplicative form, see \cite{Ba}) is the statement that
Lebesgue measure is log-concave, while this was generalized in 1975 by
Borell \cite{Bor} who proved that an absolutely continuous probability
measure is log-concave if and only if its density is log-concave. Perhaps
the most important log-concave functions are the indicator functions of
convex bodies.\ Log-concave functions are not only a functional
generalization of convex bodies, they are of paramount importance in the
study of convex bodies. An excellent example of this is Klartag's proof of
the central limit theorem for convex bodies \cite{Kl}. One of the key
properties of log-concave measures is that the measure projection of a
log-concave measure onto a linear subspace is log-concave (this is a
consequence of the Pr\'{e}kopa-Leindler inequality, see \cite{Ba}). In
particular, if $x$ is a random vector in $\mathbb{R}^{d}$ with a log-concave
distribution and $y\in \mathbb{R}^{d}$ is any fixed vector, then $%
\left\langle x,y\right\rangle $ has a log-concave distribution in $\mathbb{R}
$.\ Hence, results on log-concave measures often reduce to the one
dimensional case.

Just as for a convex body, a log-concave probability measure $\mu $ is
called isotropic if its centroid lies at the origin and its covariance
matrix is of the form%
\[
cov(\mu )=L_{\mu }^{2}I_{d}
\]%
Log-concave measures are very rigid. \ A classic example of this rigidity
(see lemma 5.5. in \cite{LV} and p. 1913 in \cite{Bob}) is the following:

\begin{lemma}
For any log-concave probability density function $f$ defined on $\mathbb{R}$
with mean zero and variance one,%
\begin{eqnarray}
1/8 &\leq &f(0)\leq \sup_{t\in \mathbb{R}}f(t)\leq 1  \label{log rigid} \\
\frac{1}{\sqrt{12}} &\leq &f(m)\leq \frac{1}{\sqrt{2}}  \nonumber
\end{eqnarray}%
where $m$ is the median of $f$.
\end{lemma}

A direct consequence of the above lemma is that if $\mu $ is an isotropic
log-concave probability measure on $\mathbb{R}^{d}$ with density $f$ and $%
\mathcal{H}$ is any hyperplane containing the origin, then%
\begin{equation}
\frac{1}{8}L_{\mu }^{-1}\leq \int_{\mathcal{H}}f(x)dx\leq L_{\mu }^{-1}
\label{max}
\end{equation}%
In particular, all central sections have roughly the same volume and the
specific hyperplane $\mathcal{H}$ mentioned in the hyperplane conjecture is
actually irrelevant to the problem. It also demonstrates the equivalence
between the hyperplane conjecture and the claim $L_{K}<c$. To see why the
bound (\ref{max}) follows from (\ref{log rigid}), consider the case $L_{\mu
}=1$ and let $X$ be a random vector with distribution $\mu $. For any $%
\theta \in S^{d-1}$ the random variable $\left\langle X,\theta \right\rangle 
$ has mean zero, variance one and a density given by 
\[
f_{\theta }(t)=\int_{\left\langle x,\theta \right\rangle =t}f(x)dx
\]%
Using similar reasoning and a higher dimensional version of (\ref{log rigid}%
), see e.g. lemma 5.14 in \cite{LV}, it is easy to prove that if $\mathcal{H}
$ is any affine subspace of dimension $k$ containing the centroid of $\mu $
then%
\[
c_{1}(n,k)L_{\mu }^{-(n-k)}\leq \int_{\mathcal{H}}f(x)dx\leq
c_{2}(n,k)L_{\mu }^{-(n-k)}
\]%
where $c_{1}(n,k)=2^{-7(n-k)}$ and $c_{2}(n,k)=(n-k)(20(n-k))^{(n-k)/2}$.
This is a slight generalization of a result by Hensley \cite{He} (see also
the discussion in \cite{Ba0} and \cite{Bob2}).

Another rigidity property (see lemma 5.4 in \cite{LV}) is that for any
log-concave probability density function $f$ defined on $\mathbb{R}$ with
mean zero,%
\[
e^{-1}\leq \int_{0}^{\infty }f(x)dx\leq 1-e^{-1}
\]%
Equivalently $F^{-1}(1-e)\geq 0$, where $F$ is the cumulative distribution
corresponding to $f$. Piecing together known results, we easily prove the
following extension:

\begin{lemma}
Let $\rho \in (0,e^{-1})$ and let $\mu $ be an absolutely continuous
log-concave probability measure on $\mathbb{R}$ with mean zero, variance $%
\sigma ^{2}$ and cumulative distribution $F$.\ Then,%
\begin{equation}
(e^{-1}-\rho )\sigma \leq F^{-1}(1-\rho )\leq 10\log (2\rho ^{-1})\sigma
\label{b}
\end{equation}
\end{lemma}

\begin{proof}
We can assume without loss of generality that $\sigma =1$. By the result
cited above (lemma 5.4 in \cite{LV}), $F^{-1}(1-\rho )>0$.\ Let $x$ be a
random variable with distribution $\mu $ and density $f$.\ Lemma 5.5 in \cite%
{LV} states that $f(0)\geq 1/8$ and for all $t\in \mathbb{R}$, $f(t)\leq 1$.
\ Hence (again by lemma 5.4 in \cite{LV})%
\begin{eqnarray*}
e^{-1}-\rho &\leq &\mu \{y\in \mathbb{R}:0\leq y\leq F^{-1}(1-\rho )\} \\
&=&\int_{0}^{F^{-1}(1-\rho )}f(s)ds \\
&\leq &F^{-1}(1-\rho )
\end{eqnarray*}%
Since $g=-\log f$ is convex, $f=e^{-g}$ decays exponentially (or quicker).
In particular (see e.g. lemma 2.2. in \cite{Kl}) for all $t\geq 0$,%
\[
1-F(t)\leq 2e^{-t/10}
\]%
from which the upper bound follows.
\end{proof}

This easily transfers to the multidimensional setting.\ For any half-space $%
\mathfrak{H}$ containing the centroid of a log-concave probability measure $%
\mu $, $\mu (\mathfrak{H})\geq e^{-1}$ (see lemma 5.12 in \cite{LV}).
Equivalently, any half-space of mass less than $e^{-1}$ can not contain the
centroid.\ The above lemma can be re-cast as follows.

\begin{lemma}
Let $\rho \in (0,e^{-1})$ and let $\mu $ be an isotropic log-concave
probability measure on $\mathbb{R}^{d}$.\ If $\mathfrak{H}_{\theta
,t}=\{x\in \mathbb{R}^{d}:\left\langle \theta ,x\right\rangle \geq t\}$ is
any half-space with $\mu (\mathfrak{H}_{\theta ,t})=\rho $ then%
\begin{equation}
(e^{-1}-\rho )L_{\mu }\leq t\leq 10\log (2\rho ^{-1})L_{\mu }  \label{a}
\end{equation}
\end{lemma}

\begin{proof}[Proof of theorem 1]
Denote the hyperplane conjecture as $\Lambda _{\mathcal{H}}$. Clearly $%
\Lambda _{2}\Rightarrow \Lambda _{1}$. We now demonstrate that $\Lambda
_{1}\Rightarrow \Lambda _{\mathcal{H}}\Rightarrow \Lambda _{2}$.

Suppose that $\Lambda _{1}$ is true and let $\delta ^{\prime }=(\delta
+e^{-1})/2$. Consider any $d\in \mathbb{N}$ and any isotropic convex body $K$
of unit volume.\ By (\ref{a}), if $\mathfrak{H}_{\theta ,t}=\{x\in \mathbb{R}%
^{d}:\left\langle \theta ,x\right\rangle \geq t\}$ is any half-space with $%
\mathrm{vol}_{d}(\mathfrak{H}_{\theta ,t}\cap K)=\delta $, then $t\geq
L_{\mu }(e^{-1}-\delta )$. \ Hence $\mathfrak{H}_{\theta ,t}\cap B(0,L_{\mu
}(e^{-1}-\delta ^{\prime }))=\emptyset $.\ Since this holds for any such
hyperplane, $B(0,L_{\mu }(e^{-1}-\delta ^{\prime }))\subset K_{\delta }$. \
By $\Lambda _{1}$, $K_{\delta }\subset B(0,r)$. \ Hence $L_{\mu }\leq
r(e^{-1}-\delta ^{\prime })^{-1}$. \ Since this holds for any isotropic
convex body in any dimension, this implies the truth of $\Lambda _{\mathcal{H%
}}$.

Suppose that $\Lambda _{\mathcal{H}}$ is true. In particular, $L_{K}<c$ for
any convex body $K$. Consider any $\delta \in (0,e^{-1})$ and let $r=10\log
(2\delta ^{-1})c$. Let $K$ be an isotropic convex body of unit volume in $%
\mathbb{R}^{d}$ and let $\mathfrak{H}_{\theta ,t}=\{x\in \mathbb{R}%
^{d}:\left\langle \theta ,x\right\rangle \geq t\}$ denote any half-space
with $\mathrm{vol}_{d}(\mathfrak{H}_{\theta ,t}\cap K)=\delta $. By (\ref{a}%
), $t\leq r$. For any $\theta \in S^{n-1}$ such a half-space exists, hence $%
K_{\delta }\subset B(0,r)$ and $\Lambda _{2}$ holds true.
\end{proof}

\bigskip

\section{Small perturbations}

\begin{proof}[Proof of theorem 2]
We may assume without loss of generality that $K$ is isotropic.\ In
particular, the center of mass of $K$ is zero.\ Fix any $\theta \in S^{d-1}$
and $t\in \lbrack m_{\theta },h_{K}(\theta )]$.\ Brunn's theorem \cite{Ba}
claims that the function $s\mapsto \psi _{\theta }(s)^{1/(d-1)}$ is concave
on its support. For any $s\in (t,h_{K}(\theta ))$, we have the convex
combination%
\[
s=\frac{h_{K}(\theta )-s}{h_{K}(\theta )-t}t+\frac{s-t}{h_{K}(\theta )-t}%
h_{K}(\theta )
\]%
Since $\psi _{\theta }\geq 0$, it follows by concavity that%
\[
\psi _{\theta }(s)\geq \left( \frac{h_{K}(\theta )-s}{h_{K}(\theta )-t}%
\right) ^{d-1}\psi _{\theta }(t)
\]%
and by integration that%
\begin{equation}
A_{\theta }(t)\geq \alpha _{\theta }(t)\psi _{\theta }(t)  \label{primary}
\end{equation}%
where $\alpha _{\theta }(t)=d^{-1}(h_{K}(\theta )-t)>0$. However, for any $%
s\in (m_{\theta },t)$ we have the convex combination%
\[
t=\frac{h_{K}(\theta )-t}{h_{K}(\theta )-s}s+\frac{t-s}{h_{K}(\theta )-s}%
h_{K}(\theta )
\]%
We again have by concavity that%
\[
\psi _{\theta }(t)\geq \left( \frac{h_{K}(\theta )-t}{h_{K}(\theta )-s}%
\right) ^{d-1}\psi _{\theta }(s)
\]%
hence%
\begin{eqnarray*}
1/2-A_{\theta }(t) &=&\int_{m_{\theta }}^{t}\psi _{\theta }(s)ds \\
&\leq &\beta _{\theta }(t)\psi _{\theta }(t)
\end{eqnarray*}%
where $\beta _{\theta }(t)=d^{-1}[h_{K}(\theta )-t]^{1-d}\left(
[h_{K}(\theta )-m_{\theta })]^{d}-[h_{K}(\theta )-t]^{d}]\right) >0$. Thus,%
\begin{equation}
A_{\theta }(t)\geq 1/2-\beta _{\theta }(t)\psi (t)  \label{secondary}
\end{equation}%
\ For large values of $\psi _{\theta }(t)$, (\ref{primary}) is a better
bound, while for small values of $\psi _{\theta }(t)$ (\ref{secondary}) is
better. Minimizing the function $u\mapsto \max \{\alpha u;1/2-\beta u\}$
gives%
\begin{eqnarray}
A_{\theta }(t) &\geq &\frac{1}{2}\alpha _{\theta }(t)[\alpha _{\theta
}(t)+\beta _{\theta }(t)]^{-1}  \nonumber \\
&=&\frac{1}{2}\left( \frac{h_{K}(\theta )-t}{h_{K}(\theta )-m_{\theta }}%
\right) ^{d}  \label{tail bound}
\end{eqnarray}%
Let $t_{\theta }=A_{\theta }^{-1}(\delta )$. Note that $\mathrm{vol}%
_{d}(K\cap \mathfrak{H}_{\theta ,t_{\theta }})=\delta $ and that $%
h_{K}(\theta )-m_{\theta }\leq 2h_{K}(\theta )$. \ By (\ref{tail bound}),%
\begin{eqnarray*}
t_{\theta } &\geq &h_{K}(\theta )-(2\delta )^{1/d}(h_{K}(\theta )-m_{\theta
}) \\
&\geq &h_{K}(\theta )(1-4\delta ^{1/d})
\end{eqnarray*}%
This implies that $\mathfrak{H}_{\theta ,t_{\theta }}\cap (1-4\delta
^{1/d})K=\emptyset $. Since this holds for all $\theta \in S^{d-1}$, we have
the inclusion $(1-4\delta ^{1/d})K\subset K_{\delta }$, and the bounds (\ref%
{geometric}) and (\ref{Banach-Mazur}) hold.
\end{proof}

Inequality (\ref{geometric}) is essentially sharp. To see this, consider the
simplex in standard orthogonal position,%
\[
\Delta _{c}^{d}=\{x\in \mathbb{R}^{d}:\forall i\text{ }x_{i}\geq
0,\sum_{i=1}^{d}x_{i}\leq 1\}
\]%
and the half-space $\mathfrak{H}=\{x\in \mathbb{R}^{d}:x_{1}\leq 1-\delta
^{1/d}\}$. Since $\Delta _{c}^{d}\backslash \mathfrak{H}$ is homothetic to $%
\Delta _{c}^{d}$ by a factor of $\delta ^{1/d}$, $\mathrm{vol}_{d}(\mathfrak{%
H}\cap \Delta _{c}^{d})=(1-\delta )\mathrm{vol}_{d}(\Delta _{c}^{d})$. \
Hence, if $K=\Delta _{c}^{d}$, then $K_{\delta }\subset \mathfrak{H}$ and%
\begin{equation}
d_{\mathfrak{L}}(K,K_{\delta })\geq 1+\frac{1}{2}\delta ^{1/d}
\label{simplex bound}
\end{equation}

\bigskip

\section{Large perturbations}

In high dimensional spaces, the mass inside an isotropic convex body $K$ is
roughly normally distributed in most directions. The measure of 'most
directions' is Haar measure on $S^{d-1}$. \ For the Euclidean ball this is
referred to as Maxwell's principal, while the general case was proved by
Klartag \cite{Kl}. A consequence of this is that most of the one dimensional
marginals of the uniform probability measure on $K$ have tails that are very
long and very light. Even to cut off a small proportion of the mass of $K$
in a fixed direction requires us to go very deeply inside the body. Thus the
body $K_{\delta }$ may be much smaller than the original body $K$. \ 

Indeed, as was pointed out to myself by Professor Mathieu Meyer, for fixed $%
\delta >0$ the effect of the operation $K\mapsto K_{\delta }$ is,
asymptotically in dimension, worst possible. By a result in \cite{MP}, for
all $\delta >0$ there exists a universal constant $C>1$ such that for any $%
d\in \mathbb{N}$ and any centrally symmetric convex body $K\subset \mathbb{R}%
^{d}$, $d_{BM}(K_{\delta },B_{2}^{d})\leq C$. Hence, there will always be
some symmetric convex body $K$ (in particular the cube $B_{\infty }^{d}$ and
the cross polytope $B_{1}^{d}$) such that $d_{BM}(K,K_{\delta })\approx 
\sqrt{d}$. \ Both the size and the shape of $K$ are dramatically altered.

\begin{proof}[Proof of theorem 3]
The relationship between $K$ and $K_{\delta }$ is independent of Euclidean
structure and we can therefore assume that $K$ is in isotropic position.
Define $v_{d}=\mathrm{vol}_{d}(B_{2}^{d})=\pi ^{d/2}(\Gamma (\frac{d}{2}%
+1))^{-1}$. Hence the body $v_{d}^{-1/d}B_{2}^{d}$ has unit volume. By
Stirling's formula, $v_{d}^{-1/d}\geq c^{\prime }\sqrt{d}$. Since $\mathrm{%
vol}_{d}(K)=1$, it can not be a proper subset of $v_{d}^{-1/d}B_{2}^{d}$,
and there exists $y\in K$ with $||y||_{2}\geq c^{\prime }\sqrt{d}$. As in
the proof of theorem 1, we have $K_{\delta }\subset 10L_{K}\log (2\delta
^{-1})B_{2}^{d}\subset cd^{1/4}\log (2\delta ^{-1})B_{2}^{d}$. Inequality (%
\ref{lower bound 2}) then follows. \ 
\end{proof}

\end{document}